\documentclass[a4paper]{article}
\usepackage[margin=25mm]{geometry}

\usepackage{authblk}
\usepackage{amsmath}
\usepackage{amssymb}
\usepackage{amsthm}
\usepackage{amscd}
\usepackage{mathtools}
\usepackage{hyperref}
\usepackage{cleveref}

\providecommand{\keywords}[1]
{
  \small	
  \textbf{\textit{Keywords---}} #1
}

\newtheorem{theorem}{Theorem}[section]

\newtheorem{proposition}[theorem]{Proposition}
\newtheorem{lemma}[theorem]{Lemma}
\theoremstyle{definition}

\theoremstyle{remark}

\newcommand\R[1]{\mathbb{R}^{#1}}
\newcommand{\pr}{\mathrm{pr}}
\newcommand{\id}{\mathrm{id}}
\newcommand{\Vor}{\operatorname{Vor}}
\newcommand{\Del}{\operatorname{Del}}
\newcommand{\del}{\operatorname{del}}
\newcommand{\Cech}{\operatorname{\textnormal{\v{C}}}}
\newcommand{\DC}{\operatorname{\textnormal{Del\v{C}}}}
\newcommand{\dC}{\operatorname{\textnormal{del\v{C}}}}

\title{Relative persistent homology}
\date{}

\author[1]{Nello Blaser}
\affil[1]{Department of Informatics, University of Bergen,
  Norway}
\affil[]{\textit {nello.blaser@uib.no}}
\author[2]{Morten Brun}
\affil[2]{Department of Mathematics, University of Bergen,
  Norway}
\affil[]{\textit{morten.brun@uib.no}}

\begin{document}

\maketitle

\begin{abstract}
  The alpha complex efficiently computes persistent homology of a
  point cloud \(X\)
  in Euclidean space when the dimension \(d\)
  is low. Given a subset \(A\)
  of \(X\),
  relative persistent homology can be computed as the persistent
  homology of the relative \v{C}ech complex \(\Cech(X, A)\).
  But this is not computationally feasible for larger point clouds
  \(X\).
  The aim of this note is to present a method for efficient
  computation of relative persistent homology in low dimensional
  Euclidean space. We introduce the relative Delaunay \v{C}ech complex
  \(\DC(X, A)\)
  whose homology is the relative persistent homology. It can be
  constructed from the Delaunay complex of an embedding of the point
  clouds in \((d+1)\)-dimensional Euclidean space.
\end{abstract}
\hspace{10pt}

{\keywords{topological data analysis, relative homology,
  Delaunay-\v{C}ech complex, alpha complex}}

\section{Introduction}\label{sec:intro}
Persistent homology is receiving growing attention in the machine
learning community. In that light, the scalability of persistent
homology computations is of increasing importance. To date, the alpha
complex is the most widely used method to compute persistent homology
for large low-dimensional data sets. 

Relative persistent homology has been considered several times in
recent years. For example Edelsbrunner and Harrer
\cite{Edelsbrunner2008} have presented an application of relative
persistent homology to estimate the dimension of an embedded
manifold. Relative persistent homology is also a way to introduce the
concept of extended persistence \cite{CohenSteiner2009}.  De Silva and
others have shown that the relative persistent homology
\(H_{\ast}(X, A_t)\) with an increasing family of sets \(A_t\) and a
constant \(X = \cup_t A_t\), and the corresponding relative persistent
cohomology have the same barcode \cite{deSilva2011}. They also show
that absolute persistent homology of \(A_t\) can be computed from this
particular type of relative persistent homology. More recently,
Pokorny and others \cite{Pokorny2016} have used relative persistent
homology to cluster two-dimensional trajectories. Some software, such
as PHAT \cite{Bauer2017}, even allows for the direct computation of
relative persistent homology. For an example see the
\href{https://github.com/blazs/phat/blob/master/src/relative_example.cpp}{PHAT
  github repository}.

Despite the fact that relative persistent homology has been considered
in many different situations, we are not aware of a relative version
of the alpha- or Delaunay \v Cech complexes being used. 

Our
contributions are as follows.
\begin{enumerate}
\item We give a new elementary proof that the Delaunay \v Cech complex
  is homotopy equivalent to the \v Cech complex. This has
  previously been shown using discrete Morse theory \cite{MR3605986}.
\item We extend this proof to the relative versions of the Delaunay \v
  Cech complex and the \v Cech complex.
\item We explain how the relative Delaunay \v Cech complex can be
  computed through embedding in a higher dimension.
\end{enumerate}

Together, these contributions result in \cref{thm:main}, which shows
how the relative persistent homology of \v{C}ech persistence modules
\(\Cech_*(X; k)/ \Cech_*(A; k)\) of low-dimensional spaces can be
efficiently computed using a relative Delaunay-\v{C}ech complex.

\begin{theorem} \label{thm:main} Let
  \(A \subseteq X \subseteq \R{d}\) be finite.  The relative
  Delaunay-\v{C}ech complex \(\DC(X, A)\) defined in
  \cref{sec:computedelcech} is homotopy equivalent to the relative
  \v{C}ech complex \(\Cech(X, A)\).

  Moreover, given the cardinalities \(n_X\) of \(X\) and \(n_A\) of
  \(A\), the relative 
  Delaunay-\v{C}ech complex contains at most
  \(O\left((n_X+n_A) \lceil (d+1)/2 \rceil\right)\) simplices.
\end{theorem}

This manuscript is structured as follows. In
\cref{sec:relativepersistence}, we introduce relative persistent
homology. \cref{sec:dowkernerves} introduces Dowker Nerves, the
theoretical foundation we use to prove that the relative Delaunay
\v{C}ech complex is homotopy equivalent to the relative \v{C}ech
complex. In \cref{sec:delaunaycech}, we introduce the alpha- and
Delaunay-\v{C}ech complexes using the Dowker Nerve notation and show
that they are homotopy equivalent to the \v{C}ech
complex. \cref{sec:reldelaunay} introduces
the relative alpha- and Delaunay-\v{C}ech complexes, and proves that
they are homotopy equivalent to the relative \v Cech complex.
Finally, in \cref{sec:computedelcech} we show how the relative
Delaunay-\v{C}ech complex can actually be computed.
  
\section{Relative Persistent Homology}
\label{sec:relativepersistence}

Let \(X\) be a finite subset of Euclidean space \(\R{d}\). Given
\(t > 0\), the \v Cech complex \(\Cech_t(X)\) of \(X\) is the abstract
simplicial complex with vertex set \(X\) and with
\(\sigma \subseteq X\) a simplex of \(\Cech_t(X)\) if and only if
there exists a point \(p \in \R{d}\) with distance less than \(t\) to
every point in \(\sigma\). Varying \(t\) we obtain the filtered \v
Cech complex \(\Cech(X)\).

Given a subset \(A\) of \(X\) we obtain an inclusion
\(\Cech(A) \subseteq \Cech(X)\) of filtered simplicial complexes and
an induced inclusion \(\Cech_*(A; k) \subseteq \Cech_*(X; k)\) of
associated chain complexes of persistence modules over the field
\(k\). The {\em relative persistent homology of the pair \((X,A)\)} is
defined as the homology of the factor chain complex of persistence
modules \(\Cech_*(X; k)/ \Cech_*(A; k)\).

For \(X\) of small cardinality, the relative persistent homology can
be calculated as 
the reduced persistent homology of the relative \v{C}ech complex
\(\Cech(X, A)\), where \(\sigma \subseteq X\) is a simplex of
\(\Cech(X, A)_t\) if either \(\sigma \subseteq A\) or
\(\sigma \in \Cech_t(X)\).  However, as the cardinality of \(X\) grows, this
quickly becomes computationally infeasible.

\section{Dowker Nerves}
\label{sec:dowkernerves}

A {\em dissimilarity} is a continuous function of the form
\(\Lambda \colon X \times Y \to [0,\infty]\), for topological spaces
\(X\) and \(Y\), where \([0,\infty]\) is given the order topology.  A
{\em morphism} \(f \colon \Lambda \to \Lambda'\) of dissimilarities
\(\Lambda \colon X \times Y \to [0,\infty]\) and
\(\Lambda' \colon X' \times Y' \to [0,\infty]\) consists of a pair
\((f_1,f_2)\) of continuous functions \(f_1 \colon X \to X'\) and
\(f_2 \colon Y \to Y'\) so that for all \((x,y) \in X \times Y\) the
following inequality holds:
\begin{displaymath}
  \Lambda'(f_1(x), f_2(y)) \le \Lambda(x,y).
\end{displaymath}
This notion of morphism is less general than the definition in for
example \cite[Definition 2.10]{SFN}, but it is simpler and suffices
for our purposes.

The {\em Dowker Nerve} \(N\Lambda\) of \(\Lambda\) is the filtered
simplicial complex described as follows: For \(t > 0\), the simplicial
complex \(N\Lambda_t\) consists of the finite subsets \(\sigma\) of \(X\) for
which there exists \(y \in Y\) so that \(\Lambda(x,y) < t\) for every
\(x \in \sigma\).

Let \(f \colon \Lambda \to \Lambda'\) be a morphism of
dissimilarities as above and let \(\sigma \in N\Lambda_t\). Given \(y
\in Y\) with \(\Lambda(x,y) < t\) for every \(x \in \sigma\) we see
that
\begin{displaymath}
  \Lambda'(f_1(x), f_2(y)) \le \Lambda(x,y) < t.
\end{displaymath}
for every \(x \in \sigma\), so \(f_1(\sigma) \in N\Lambda'_t\). Thus
we have a simplicial map \(f \colon N\Lambda \to N\Lambda'\).

Given \(x \in X\) and \(t > 0\), the \emph{\(\Lambda\)-ball of radius \(t\)
centered at \(x\)} is the subset of \(Y\) defined as
\begin{displaymath}
  B_{\Lambda}(x, t) = \{y \in Y,\ \mid \, \Lambda(x,y) < t\}.
\end{displaymath}
The {\em \(t\)-thickening of \(\Lambda\)} is the subset of \(Y\) defined as
\begin{displaymath}
  \Lambda^t = \bigcup_{x \in X} B_{\Lambda}(x, t).
\end{displaymath}
Note that by construction the set of \(\Lambda\)-balls of radius \(t\)
is an open cover of the \(t\)-thickening of \(\Lambda\).

The {\em geometric realization} \(|K|\) of a simplicial complex
\(K\) on the vertex set \(V\) is the subspace of the space \([0,1]^V\)
of functions \(\alpha 
\colon V \to [0,1]\) described as follows:
\begin{enumerate}
\item The subset \(\alpha^{-1}((0,1])\) of \(V\) 
consisting of elements where \(\alpha\) is
  strictly positive is a simplex in \(K\). In particular it is finite.
\item The sum of the values of \(\alpha\)
  is one, that is \(\sum_{v \in V} \alpha(v) = 1\). 
\end{enumerate}

The subspace topology on \(|K|\) is called the {\em strong topology} on the
geometric realization. It is 
convenient for construction of functions into \(|K|\). The {\em weak
  tooplogy} on \(|K|\), which we are not going to use here, is
convenient for construction of functions out of \(|K|\). The homotopy
type of \(|K|\) is the same for these two topologies
\cite[p. 355, Coorllary A.2.9]{Dold}.
Given a simplex \(\sigma \in K\), the simplex \(|\sigma|\) of \(|K|\)
is the closure of 
\[\{\alpha \colon V \to
[0,1]\, | \, \alpha(v) > 0 \text{ for all \(v \in \sigma\)}\}.\]
The simplices of \(|K|\) are the sets of this form. 

A {\em partition of unity subordinate to the dissimilarity
  \(\Lambda \colon X \times Y \to [0,\infty]\)} consists of continuous
maps
\(\varphi^t \colon \Lambda^t \to |N\Lambda_t|\) such that given
\(x \in X\), the closure of the set
\begin{displaymath}
  \{y \in Y \, \mid \, \varphi^t(y)(x) > 0\}
\end{displaymath}
is contained in \(B_{\Lambda}(x,t)\).  We say that \(\Lambda\) is {\em
  numerable} if a partition of unity subordinate to \(\Lambda\)
exists. If \(Y\) is paracompact, then every dissimilarity of the form
\(\Lambda \colon X \times Y \to [0, \infty]\) is numerable
\cite[p. 355, paragraph after Definition A.2.10]{Dold}.

Let \(y \in \Lambda^t\) and let \(\{\varphi^t \colon \Lambda^t \to
|N\Lambda_t|\}\) be a partition 
of unity subordinate to \(\Lambda\). If \(x \in X\) with
\(\varphi^t(y)(x) > 0\), then \(\Lambda(x,y) < t\). Therefore
\(\varphi^t(y)\) is 
contained in a simplex \(|\sigma|\) in \(|N\Lambda_t|\) with
\(\sigma\) contained 
in \(\{x \in X \, \mid \, \Lambda(x,y) < t\}\). Every
finite subset of this set is 
an element of \(N \Lambda_t\).
This implies that for \(s \le t\) there is a
simplex of \(|N\Lambda_t|\) containing both \(\varphi^s(y)\) and
\(\varphi^t(y)\). It also implies that given another partition of
unity \(\{\psi^t \colon \Lambda^t \to |N\Lambda_t|\}\) subordinate to
\(\Lambda\) there is a simplex of \(|N\Lambda_t|\) containing both
\(\varphi^t(y)\) and \(\psi^t(y)\). This is exactly the definition of
contiguous maps, so \(\varphi^t\) and \(\psi^t\) are contiguous, and
thus homotopic maps \cite[Remark 2.22, p. 350]{Dold}.  Similarly, the
diagram
\begin{displaymath}
  \begin{CD}
    \Lambda^s @>{\varphi^s}>> |N\Lambda_s| \\
    @VVV @VVV\\
    \Lambda^t @>{\varphi^t}>> |N\Lambda_t|
  \end{CD}
\end{displaymath}
commutes up to homotopy \cite[paragraph on the nerve starting on page
355 and ending on page 356]{Dold}.

Recall that a cover \(\mathcal U\) of \(Y\) is {\em good} if all
non-empty finite intersections of members of \(\mathcal U\) are
contractible. We now state the Nerve Lemma in the context of dissimilarities.
\begin{theorem} \label{thm:partitionofunity} 
  If \(Y\) is paracompact, then there exists a partition of unity
  \(\{\varphi^t\colon \Lambda^t \to |N\Lambda_t|\}\) subordinate to every dissimilarity
  \(\Lambda \colon X \times Y \to [0,\infty]\).  Moreover, if the
  cover of \(\Lambda^t\) by \(\Lambda\)-balls of radius \(t\) is a
  good cover, then \(\varphi^t\) is a homotopy equivalence.
\end{theorem}
\begin{proof}
  By the above discussion, we only need to note that the last statement
  about good covers is \cite[Theorem 4.3]{virk2019rips}.
\end{proof}
A functorial version of the Nerve Lemma can be stated as follows:
\begin{proposition} \label{prop:realizationmorphismcommute}
  Let \(\Lambda \colon X \times Y \to [0,\infty]\) and 
  \(\Lambda' \colon X' \times Y' \to [0,\infty]\) be dissimilarities
  and let \(f = f_1 \times f_2 \colon X \times Y \to X' \times Y'\) be
  a morphism \(f \colon \Lambda \to \Lambda'\) of dissimilarities. If  
  \(\{\varphi^t\colon \Lambda^t \to |N\Lambda_t|\}\)
  is a partition of
  unity  subordinate to \(\Lambda\) and 
  \(\{\psi^t\colon (\Lambda')^t \to |N\Lambda'_t|\}\)
  is a partition of
  unity  subordinate to \(\Lambda'\), then for every \(t \ge 0\) the diagram
  \begin{displaymath}
    \begin{CD}
      \Lambda^t @>{\varphi^t}>> |N\Lambda_t| \\
      @V{f_2}VV @VV{|f_1|}V\\
      (\Lambda')^t @>{\psi^t}>> |N\Lambda'_t|,
    \end{CD}    
  \end{displaymath}
  commutes up to homotopy.
\end{proposition}
\begin{proof}
  We show that the two compositions are contiguous. Recall that
  \(|f_1|\) takes a point \(\alpha \colon X \to [0,1]\) of
  \(|N\Lambda_t|\) to the point \(|f_1|(\alpha)\) of \(|N\Lambda'_t|\)
  with \(|f_1|(\alpha)(x') = \sum_{f_1(x) = x'} \alpha(x)\).  Recall
  further that \(\varphi^t(y)\) 
  is contained in a simplex \(|\sigma|\) in \(|N\Lambda_t|\), where
  \(\sigma\) is contained 
  in \(\{x \in X \, \mid \, \Lambda(x,y) < t\}\).
  Then we have that for 
  \(y \in \Lambda^t\), the elements \(|f_1|(\varphi^t(y))\) and
  \(\psi^t(f_2(y))\) of \(|N\Lambda'_t|\) are contained in simplices
  \(|\sigma|\) and \(|\tau|\) respectively. Both \(\sigma\) and
  \(\tau\) are subsets of the set
  \(\{x' \in X' \, \mid \, \Lambda'(x',f_2(y) ) < t\}\). However
  every finite subset of this set is a simplex in \(N\Lambda'_t\). In
  particular, so is the union \(\sigma \cup \tau\).
\end{proof}

\section{The Alpha- and Delaunay \v Cech Complexes}
\label{sec:delaunaycech}

Given a finite subset \(X\) of \(\R{d}\) we define the
Voronoi cell of \(x \in X\) as 
\begin{displaymath}
  \Vor(X,x) = \{p \in \R{d} \, \mid \, d(x,p) \le d(y, p) \text{ for all
    \(y \in X\)}\}.
\end{displaymath}

Let \(\R{d}_d\) be Euclidean space with the discrete
topology.
The {\em discrete Delaunay dissimilarity} of \(X\) is defined as
\begin{displaymath}
  \del^X \colon X \times \R{d}_d \to [0, \infty], \quad
  \del^X(x,p) =
  \begin{cases}
    0 & \text{if \(p \in V(X,x)\)} \\
    \infty & \text{if \(p \notin V(X,x)\)}.
  \end{cases}
\end{displaymath}
The {\em Delaunay complex \(\Del(X)\)} is the simplicial
complex with vertex set \(X\) and with \(\sigma \subseteq X\) a
simplex of \(\Del(X)\) if and only if there exists a point in
\(\R{d}\) belonging to \(\Vor(X,x)\) for every
\(x \in \sigma\). That is, \(\Del(X) = N\del^X_t\) for \(t > 0\).

Note that with respect to Euclidean topology, the discrete
Delaunay dissimilarity is not continuous, and hence \(\del^X \colon X
\times \R{d} \to [0, \infty]\) is not a
dissimilarity.  One way to deal with this is to use the Nerve Lemma
for absolute neighbourhood retracts \cite[Theorem 8.2.1]{MR387495}. In
order to use \cref{thm:partitionofunity} and
\cref{prop:realizationmorphismcommute} from above, instead we
construct a continuous version of the Delaunay dissimilarity.

Given a subset \(\sigma\) of \(X\) and \(p \in \R{d}\), let
\begin{displaymath}
  d_{\Vor}(p, \sigma) = \max \{d(p, \Vor(X,x)) \, \mid \, x \in \sigma\}, 
\end{displaymath}
where for any \(A \subseteq \R{d}\), we define
\(d(p, A) = \inf_{a \in A} \{d(p, a) \}\).

Note that if \(\sigma \notin \Del(X)\), then the infimum
\(\varepsilon_\sigma\) of the continuous function
\(d_{\Vor}(-,\sigma) \colon \R{d} \to \R{}\) is strictly positive.
Choose \(\varepsilon > 0\) so that
\(2\varepsilon < \varepsilon_{\sigma}\) for every subset \(\sigma\) of
\(X\) that is not in \(\Del(X)\). Given \(x \in X\) we define the
\(\varepsilon\)-thickened Voronoi cell \(\Vor(X,x)^\varepsilon\) by
\begin{displaymath}
  \Vor(X,x)^\varepsilon = \{p \in \R{d} \, \mid \, d(p, \Vor(X,x)) <
  \varepsilon\}. 
\end{displaymath}
By construction the nerve of the open cover
\((\Vor(X,x)^\varepsilon)_{x \in X}\) of \(\R{d}\) is equal to
\(\Del(X)\).

Let \(h \colon [0, \infty] \to [0, \infty]\) be the order preserving map
\begin{equation}  \label{voronoithickening}
  h(t) =
  \begin{cases}
    -\ln (1 - t / \varepsilon) & \text{if \(t < \varepsilon\)} \\
    \infty & \text{if \(t \ge \varepsilon\)}.
  \end{cases}
\end{equation}
For each \(x \in X\) we let \(\Del_x \colon \R{d} \to [0, \infty]\) be
the function defined by \(\Del_x(p) = h(d(p, \Vor(X,x)))\) so that
\(\Del_x(\Vor(X,x)) = 0\) and
\(\Del_x(\R{d} \setminus \Vor(X,x)^\varepsilon) = \infty\).

The Delaunay dissimilarity of \(X\) is defined as
\begin{displaymath}
  \Del^X \colon X \times \R{d} \to [0, \infty], \quad
  \Del^X(x,p) = \Del_x(p).
\end{displaymath}
By the above discussion we know that \(N\Del^X_t = N\del^X_t = \Del(X)\) whenever
\(t > 0\).

The \v Cech dissimilarity of \(X\) is defined as
\begin{displaymath}
  d^X \colon X \times \R{d} \to [0, \infty],
\end{displaymath}
where \(d^X(x,p)\) is the Euclidean distance between \(x \in X\) and
\(p \in \R{d}\).

The alpha dissimilarity of \(X\) is defined as
\begin{displaymath}
  A^X = \max(\Del^X, d^X) \colon X \times \R{d} \to [0,\infty].
\end{displaymath}

The Delaunay \v Cech dissimilarity is defined as
\begin{displaymath}
  \DC^X \colon X \times \left(\R{d} \times \R{d}\right) \to [0, \infty], \quad
  \DC^X(x,(p,q)) = \max(d^X(x,p), \Del^X(x,q)).
\end{displaymath}

Note the nerve of the dissimilarity 
\begin{displaymath}
  \dC^X \colon X \times \left(\R{d} \times \R{d}_d\right) \to [0, \infty], \quad
  \dC^X(x,(p,q)) = \max(d(^Xx,p), \del^X(x,q))
\end{displaymath}
is identical to the nerve of \(\DC^X\). Moreover, the Dowker nerves of
the Delaunay-, \v Cech-, alpha- and Delaunay \v Cech dissimilarities
are the Delaunay-, \v Cech-, alpha- and Delaunay \v Cech complexes
respectively. For all these dissimilarities, the corresponding balls are convex, so
the geometric realizations are homotopy equivalent to the
corresponding thickenings. In order to see that the morphism
\(A^X \to d^X\) of dissimilarities induces homotopy equivalences
\(|NA^X_t| \xrightarrow \simeq |Nd^X_t|\) it suffices to note that the
corresponding map
\((A^X)^t \to (d^X)^t\)
is the identity map. This holds because
\(B_{A^X}(x,t) = B_{d^X}(x,t) \cap B_{\Del^X}(x,t)\) and given
\(y \in B_{d^X}(x,t)\) we have that
\(y \in \Vor(X,x')\) for some \(x'\in X\)
with \(d^X(y,x')\) minimal, so \(d^X(y,x') \le d^X(y,x) < t\) and
\(y \in B_{d^X}(x',t) \cap B_{\Del^X}(x',t)\).

In order to see that the morphism
\(\DC^X \to d^X\) of dissimilarities induces homotopy  
equivalences \(|N\DC^X_t|
\xrightarrow \simeq |Nd^X_t|\) we use the
following lemma: 
\begin{lemma}
\label{linelemma}
  For every \((p,q) \in (\DC^X)^t\), the entire line segment between
  \((p,p)\) and \((p,q)\) is contained in \((\DC^X)^t\).
\end{lemma}
\begin{proof}
  In order not to clutter notation we omit superscript \(X\) on
  dissimilarities. 
  Let \(\gamma \colon [0,1] \to \R{d}\) be the function \(\gamma(s) =
  \left(p, (1-s)p + sq\right) \).
  We claim that given \((p,q) \in \DC^t\) and \(s \in [0,1]\) the point
  \((p, \gamma(s)) = \left(p, (1-s)p + sq\right)\) is in \(\DC^t\).

  If \((p, q) \in \DC^t\), there exists a point \(x \in X\), such that
  \(p \in B_d(x, t)\) and \(q \in B_{\Del}(x, t)\), that is,
  \(d(q, \Vor(X,x)) < h^{\leftarrow}(t)\), where \(h^{\leftarrow}\) is
  the generalized inverse of \(h\). Pick \(q' \in \Vor(X, x)\) so that
  \(d(q,q') < h^{\leftarrow}(t)\).  Let
  \(\gamma' \colon [0,1] \to \R{d}\) be the function
  \(\gamma'(s) = \left(p, (1-s)p + sq'\right) \).  Given
  \(s \in [0,1]\), suppose that the point
  \((p, \gamma'(s)) = \left(p, (1-s)p + sq'\right)\) is in
  \(\dC^t\). Then \(\gamma(s)\) is in \(\DC^t\) since the distance
  between \((1-s)p + sq\) and \((1-s)p + sq'\) is less than
  \(h^{\leftarrow}(t)\).
  
  We are left to show that, 
  given \(s \in [0,1]\), the point
  \((p, \gamma'(s)) = \left(p, (1-s)p + sq'\right)\) is in
  \(\dC^t\). 
  Suppose
  \(\gamma'(s) \in \Vor(X, y)\) for some \(s \in [0, 1)\) and some
  \(y \in X\). We claim that then \(p \in B_d(y, t)\).
  To see this, we may without loss of generality assume that \(y \ne
  x\).
  Let 
  \(H\) be the hyperplane in between \(x\) and \(y\), i.e.
  \[H = \{z \in X \mid d(x, z) = d(y, z)\}.\]  
  Let 
  \[H_+ = \{z \in X \mid d(x, z) \ge d(y, z)\}\]
  and   
  \[H_- = \{z \in X \mid d(x, z) \le d(y, z)\}.\]
  Since \(\gamma'(s) \in \Vor(X,y)\) we have \(\gamma'(s) \in H_+\). Since
  \(q \in \Vor(X,x)\) we have \(q \in H_-\).
  Since the line segment between \(p\) and \(q\) either is contained in
  \(H\) or intersects \(H\) at
  most once we must have \(p \in H_+\).
  That is, \(d(y,p) \le d(x,p) < t\), so \(p \in B_d(y,t)\) as claimed.
\end{proof}
By \cref{linelemma}, the inclusion
\begin{displaymath}
  (d^X)^t = \cup_{x \in X} B_{d^X}(x,t) \to \cup_{x \in X} B_{\DC^X}(x,t) =
  (\DC^X)^t, \quad p
  \mapsto (p,p)
\end{displaymath}
is a deformation retract. In particular it is a homotopy equivalence.

\section{The Relative Delaunay \v{C}ech Complex}
\label{sec:reldelaunay}

In this section we consider two subsets \(X_1\) and \(X_2\) of
d-dimensional Euclidean space \(\R{d}\).

The {\em Voronoi diagram} of a finite subset \(X\) of \(\R{d}\) is the
set of pairs of the form \((x, \Vor(X,x))\) for \(x \in X\), that is,
\begin{displaymath}
  \Vor(X) = \{(x, \Vor(X,x)) \, \mid \, x \in X\}.
\end{displaymath}
This may seem overly formal since the projection on the first factor
gives a bijection \(\Vor(X) \to X\). However, when we work with
Voronoi cells with respect to different subsets \(X_1\) and \(X_2\) of
\(\R{d}\) it may 
happen that \(\Vor(X_1,x_1) = \Vor(X_2, x_2)\) even when \(x_1 \ne
x_2\). 
The {\em Voronoi diagram} of the pair of subsets \(X_1\) and \(X_2\) of
\(\R{d}\) is the set
\begin{displaymath}
  \Vor(X_1, X_2) = \Vor(X_1) \cup
  \Vor(X_2).
\end{displaymath}
The {\em discrete Delaunay dissimilarity} of \(X_1\) and \(X_2\) is
defined as
\begin{displaymath}
  \del^{X_1,X_2} \colon \Vor(X_1, X_2) \times \R{d}_d \to [0, \infty],
  \qquad 
  \del^{X_1,X_2}((x,V), p) =
  \begin{cases}
    0 & \text{if \(p \in V\)} \\
    \infty & \text{if \(p \notin V\).}
  \end{cases}
\end{displaymath}
The simplicial
complex \(N\del_t^{X_1,X_2}\) is independent of \(t > 0\). It is 
the {\em Delaunay complex} \(\Del(X_1, X_2)\) on \(X_1\) and
\(X_2\). In order to  
describe the homotopy type of this simplicial complex we thicken the
Voronoi cells like we did in the previous section:

Given a subset \(\sigma\) of \(\Vor(X_1,X_2)\) and \(p \in \R{d}\), let
\begin{displaymath}
  d_{\Vor}(p, \sigma) = \max \{d(p,V) \, \mid \, (x,V) \in \sigma\}.
\end{displaymath}
Note that if \(\sigma \notin \Del(X_1, X_2)\), then the infimum
\(\varepsilon_{\sigma}\) of the continuous function \(d_{\Vor}(-, \sigma)
\colon \R{d} \to \R{}\) is strictly positive. Choose \(\varepsilon >
0\) so that \(2 \varepsilon < \varepsilon_{\sigma}\) for every subset
\(\sigma\) of \(\Vor(X_1,X_2)\) that is not in \(\Del(X_1,
X_2)\). Given \((x,V) \in \Vor(X_1,X_2)\) we define the
\(\varepsilon\)-thickening \(V^{\varepsilon}\) of \(V\) by
\begin{displaymath}
  V^{\varepsilon} = \{p \in \R{d} \, \mid \, d(p, V) < \varepsilon\}.
\end{displaymath}
By construction, the nerve of the open cover
\(((x,V^{\varepsilon}))_{(x,V) \in \Vor(X_1, X_2)}\) is equal to
\(\Del(X_1,X_2)\). 
The Delaunay dissimilarity \(\Del^{X_1, X_2}\) of \(X_1\) and \(X_2\)
is defined as
\begin{displaymath}
   \Vor(X_1, X_2) \times \R{d} \xrightarrow{\Del^{X_1,X_2}} [0, \infty], \qquad
  \Del^{X_1,X_2}((x, V),
  p) = h(d(p, V))
\end{displaymath}
for \(h \colon [0,\infty] \to [0, \infty]\) the order preserving map
defined in the previous section.

The inclusion \(X_1 \to \Vor(X_1, X_2)\) taking \(x \in X_1\) to
\((x, \Vor(x, X_1))\) induces a morphism of dissimilarities
\(\Del^{X_1} \to \Del^{X_1, X_2}\) and an inclusion of nerves
\(N\Del_t^{X_1} \subseteq N\Del_t^{X_1, X_2}\) for \(t > 0\).

Next, we construct the dissimilarity \(A^{X_1, X_2}\) as
\begin{displaymath}
  \Vor(X_1, X_2) \times
  \R{d} \xrightarrow{A^{X_1,X_2}} [0,\infty], \qquad ((x,V), p)
  \mapsto \max(d(x,p), \Del^{X_1, X_2}((x, V), p)).
\end{displaymath}
Also here we have an obvious inclusion \(NA^{X_1}_t \to NA^{X_1, X_2}_t\),
and the \(A^{X_1, X_2}\)-balls are convex so the nerve lemma yields a
homotopy equivalence
\begin{displaymath}
  |NA^{X_1,X_2}_t| \simeq \bigcup_{(x, V) \in \Vor(X_1, X_2)}
  B_{A^{X_1,X_2}}((x, V), t) = \bigcup_{x \in X_1 \cup X_2} B_{d^{X_1
      \cup X_2}}(x, t) = (X_1 \cup X_2)^t.
\end{displaymath}

Finally, we construct the dissimilarity \(\DC^{X_1, X_2}\)
\begin{align*}
  \Vor(X_1, X_2) \times
  (\R{d} \times \R{d}) &\xrightarrow{\DC^{X_1,X_2}} [0,\infty], \\
  ((x,V), (p, q))
  &\mapsto \max(d(x,p), \Del^{X_1, X_2}((x, V), q)) 
\end{align*}
Here again we have an obvious inclusion \(N\DC^{X_1}_t \to N\DC^{X_1, X_2}_t\),
and the \(\DC^{X_1, X_2}\)-balls are convex so the nerve lemma yields a
homotopy equivalence
\begin{displaymath}
  |N\DC^{X_1,X_2}_t| \simeq 
  (\DC^{X_1,X_2})^t 
\end{displaymath}
The following variant of \cref{linelemma} implies that 
\((\DC^{X_1,X_2})^t \)
is a deformation retract of
\((X_1 \cup X_2)^t\). 
\begin{lemma}
\label{relativelinelemma}
  For every \((p,q) \in (\DC^{X_1, X_2})^t\), the entire line segment between
  \((p,p)\) and \((p,q)\) is contained in \((\DC^{X_1, X_2})^t\).
\end{lemma}
\begin{proof}
  Given
  \((p,q) \in (\DC^{X_1, X_2})^t = (\DC^{X_1})^t \cup
(\DC^{X_2})^t\), we have \((p,q) \in (\DC^{X_i})^t\) for some
\(i \in \{1, 2\}\). Then also \((p, p)\) lies in
\((\DC^{X_i})^t\), and \cref{linelemma} proves the claim.
\end{proof}

\section{Implementation Of The Relative Delaunay \v{C}ech Complex}
\label{sec:computedelcech}

In this section we explain how the relative Delaunay complex can be
realized as a standard Delaunay complex by embedding in one dimension
higher.

We fix some notation used in this section: \(X_1 \subseteq \R{d}\) and
\(X_2 \subseteq \R{d}\) are 
finite subsets. We let \(s\) be a positive real number, we
let \(Z = X_1 \times \{s\} \cup X_2 \times \{-s\}\) and we let \(\pr
\colon \R{d+1} \to \R{d}\) be the projection omitting the 
last coordinate. 

\begin{lemma}
  The projection \(\pr \colon \R{d+1} \to \R{d}\) induces a surjection
  \begin{displaymath}
    \Vor(Z)
    \xrightarrow{g} \Vor(X_1, X_2), \qquad ((x, s), V) \mapsto (x,
    V(X_1, x)), \quad ((x, -s), V) \mapsto (x,
    V(X_2, x)),
  \end{displaymath}
  with \(\pr(V) \subseteq V(X_i,x)\) for \(x \in X_i\).  Given
  \((x, V) \in \Vor(X_1, X_2)\) the fiber \(g^{-1}((x,V))\) consists
  of all elements of \(\Vor(Z))\) of the form \(((x,a), V)\) for
  \(a \in \{\pm s\}\).
\end{lemma}
\begin{proof}
  We show that \(\pr(V) \subseteq V(X_1,x_1)\) for
  \(((x_1,s), V) \in \Vor(Z)\)
  with \(x_1 \in X_1\).
  Given \((p,r) \in V\) we have for all points of
  the form \((x_1',s)\) for \(x_1'\in X_1\) that
  \(d((p,r), (x_1,s)) \le d((p,r), (x_1',s))\). This implies that
  \(d(p,x_1) \le d(p,x_1')\), and thus \(p \in V(X_1,x_1)\). We
  conclude that \(\pr(V) \subseteq V(X_1,x_1)\). An analogous argument
  applies for elements of the form \(((x_2,-s), V)\) in
  \(\Vor(Z)\). 
\end{proof}
Let \(s_1\) be larger than the largest filtration value of the alpha
complex of \(X_1\).  Then the function
\(j_1 \colon \Vor(X_1) \to \Vor(Z)\) defined by
\(j_1(x_1,V) = ((x_1,s), V(Z, (x_1,s)))\) induces a simplicial map of nerves
\(\del(X_1) \to \del(Z)\) for all \(s > s_1\).  
Similarly, there is a simplicial map \(\del(X_2) \to \del(Z)\) for all \(s > s_2\)
when \(s_2\) is larger than all filtration values of the alpha
complex of \(X_2\).
Let
\(s(X_1,X_2) = \max(s_1, s_2)\).

Choose \(\varepsilon > 0\) satisfying the following two criteria:
\begin{enumerate}
\item \(2 \varepsilon < \varepsilon_{\sigma}\) for every subset
  \(\sigma\) of \(\Vor(X_1,X_2)\) that is not in \(\Del(X_1, X_2)\).
\item \(2 \varepsilon < \varepsilon_{\sigma}\) for every subset
  \(\sigma\) of \(\Vor(Z)\) that is not in \(\Del(Z)\).
\end{enumerate}
Let \(h \colon [0, \infty] \to [0, \infty]\) be the order
preserving map defined in \cref{voronoithickening}, and let \(\Del^Z\) and
\(\Del^{X_1, X_2}\) be constructed using \(h\). We define a new dissimilarity
\begin{displaymath}
  D \colon \Vor(Z) \times (\R{d} \times \R{d+1}) \to [0, \infty], \quad
  D((z,V), (p, q)) = \max(d(\pr(z), p), \Del^Z((z,V), q)).
\end{displaymath}
Note that the underlying simplicial complex \(\bigcup_{t>0} ND_t\) of
the nerve of \(D\) is the Delaunay complex \(\del(Z)\). The
filtration value of \(\sigma \in \del(Z)\) in
the neve of \(D\) is the filtration value of \(g(\sigma)\) in the
nerve of \(\DC^{X_1, X_2}\). 
\begin{proposition} \label{projectionequivalence}
  Let \(X_1 \subseteq \R{d}\) and \(X_2 \subseteq \R{d}\) be
  finite. Choose \(s > s(X_1, X_2)\). Then
  \(\Vor(Z) \xrightarrow{g} \Vor(X_1, X_2)\) and
  \(\id \times \pr \colon \R{d} \times \R{d+1} \to \R{d} \times
  \R{d}\) form a morphism
  \[f = (g, \id \times \pr) \colon D \to \DC^{X_1,X_2}\] 
  of
  dissimilarities inducing a homotopy equivalence
  \begin{displaymath}
    g \colon ND_t \to
    N\DC^{X_1, X_2}_t
  \end{displaymath}
  for every \(t > 0\).
\end{proposition}
\begin{proof}
  For \(i = 1,2\) the inclusion \(\pr(V) \subseteq V(X_i,x)\) for \(((x,(-1)^{i-1}s), V)
  \in \Vor(Z)\) implies that 
  \[\Del^{X_1, X_2}(g(z,V), \pr(q)) \le \Del^Z((z, V), q)\]
  for all \(((z, V), q)\in \Vor(Z)\). So we have a morphism \(f = (g,
  \id \times \pr) \colon D \to 
  \DC^{X_1,X_2}\).

  In order to show that \(g\) induces a homotopy equivalence of geometric
  realizations, by the Nerve Lemma, it suffices to show that given a
  simplex \(\sigma\) of \(N\DC^{X_1, X_2}_t\), the inverse image
  \(g^{-1}(\sigma)\) is a simplex of \(ND_t\).  Let \(p\) be a point
  in the intersection of the Voronoi cells in \(\sigma\).  Write
  \(g^{-1}(\sigma) = \tau_1 \cup \tau_2\), where \(\tau_1\) consists
  of Voronoi cells with centers at height \(s\) and \(\tau_2\)
  consists of Voronoi cells with centers at height \(-s\).  Let
  \(\sigma_1 = \{(x_1,s) \, \mid \, (x_1, V(X_1, x_1) ) \in \sigma\}\)
  and
  \(\sigma_2 = \{(x_2, -s) \, \mid \, (x_2, V(X_2, x_2) ) \in
  \sigma\}\).

  Suppose that \(\tau_2\) is empty. Then actually
  \(\sigma \in \DC^{X_1}_t\), and since \(s > s_1\) we know that
  \(j_1(\sigma) \in \del(Z)\). Since \(g \circ j_1\) is the inclusion
  of \(\Vor(X_1)\) in \(\Vor(X_1, X_2) = \Vor(X_1) \cup \Vor(X_2)\) we
  know that \(j_1(\sigma) \subseteq g^{-1}(\sigma) = \tau_1\) and that
  \(j_1(\sigma) \in ND_t\). On the other hand, since \(\tau_2\) is
  empty and \(j_1\) is injective, we know that
  \(g^{-1}(\sigma)\) has the same cardinality as \(j_1(\sigma)\), so
  they must be equal. We conclude that \(g^{-1}(\sigma)\) is a simplex
  of \(ND_t\). A similar argument applies when \(\tau_1\) is empty.

  In the remaining case where both \(\tau_1\) and \(\tau_2\) are
  nonempty, the function
  \begin{displaymath}
    f \colon \R{d+1} \to \R, \qquad f(a) = d_{\Vor}(a, \sigma_1) - d_{\Vor}(a,\sigma_2)
  \end{displaymath}
  has \(f((p,-s)) > 0\) and \(f((p,s)) < 0\). By the intermediate
  value theorem there exists \(t \in [-s,s]\) with \(f(p,t) =
  0\). Since \((p,t)\) has the same distance to all elements of
  \(\sigma_1\) and also has the same distance to all elements of
  \(\sigma_2\) we conclude that \((p,t)\) is in the intersection of
  the Voronoi cells in \(g^{-1}(\sigma) = \tau_1 \cup \tau_2\). Thus
  \(\DC^Z((z,V), p) = 0\) and \(d(\pr(z), p) < t\) for all
  \((z,V) \in g^{-1}(\sigma)\). In particular
  \(g^{-1}(\sigma) \in ND_t\).
\end{proof}

We are now ready to compute persistent homology of \(X_1 \cup X_2\)
relative to \(X_1\). The relative Delaunay-\v{C}ech complex
\(\DC(X_1 \cup X_2 , X_1)\) is the filtered simplicial complex with
\(\DC(X_1 \cup X_2 , X_1)_t = j_1(\del(X_1)) \cup ND_t\).
\begin{theorem}
  Let \(X_1 \subseteq \R{d}\) and \(X_2 \subseteq \R{d}\) be
  finite. Choose \(s > s(X_1, X_2)\). Then there is an isomorphism
  \begin{displaymath}
    (H_*(\DC(X_1 \cup X_2 , X_1)_t))_{t > 0} \cong (H_*((X_1 \cup X_2)^t, X_1^t))_{t > 0}
  \end{displaymath}
  of persistence modules.
\end{theorem}
\begin{proof}
  Since \(j_1(\del(X_1)\) is contractible, the geometric realization
  of \(\DC(X_1 \cup X_2 , X_1)_t\) is homotopy equivalent to the
  quotient space \(|\DC(X_1 \cup X_2 , X_1)_t|/|j_1(\del(X_1)|\). This
  quotient space is homeomorphic to
  \(|ND_t|/|ND_t \cap j_1(\Del(X_1))|\). By \cref{projectionequivalence}
  the map \(g \colon ND_t \to N\DC^{X_1,X_2}_t\) induces a homotopy
  equivalence of geometric realizations. Moreover \(g\) induces an
  isomorphism \(ND_t \cap j_1(\Del(X_1)) \to N\DC^{X_1}_t\). Combining
  these two statements, \(g\)
  induces a homotopy equivalence
  \(|ND_t|/|ND_t \cap j_1(\Del(X_1))| \to |N\DC^{X_1,X_2}_t|/
  |N\DC^{X_1}_t|\).
  The space
  \(|N\DC^{X_1, X_2}_t|\) is homotpy equivalent to the Euclidean
  \(t\)-thickening \((X_1 \cup X_2)^t\) of \(X_1 \cup X_2\) and
  \(|N\DC^{X_1}_t|\) is homotopy equivalent to the Euclidean
  \(t\)-thickening \(X_1^t\) of \(X_1\).
\end{proof}

Finally, we note that the size of the relative Delaunay-\v{C}ech
complex grows linearly with the sizes \(n_i\) of the finite subsets
\(X_i\).  The Delaunay triangulation of \(n\)
points in \(d\) dimensions contains at most
\(O(n \lceil d / 2 \rceil)\) simplices \cite{Seidel1995}. Since we use
the Delaunay triangulation of \(n_1 + n_2\) points in \(d+1\)
dimensions to compute the relative Delaunay-\v{C}ech complex, it
contains at most \(O((n_1+n_2) \lceil (d+1)/2 \rceil)\) simplices.
This concludes the proof of \cref{thm:main}.

\bibliographystyle{unsrt}
\bibliography{refs}

\end{document}